\documentclass[11pt,a4paper]{article}
\usepackage{amssymb,comment,amsmath,amsthm}
\usepackage{graphicx}
\usepackage{color}
\usepackage{titling}
\usepackage{url}
\usepackage{amsthm,lipsum}
\usepackage{titling}

\newcounter{foo}

\newfont{\blb}{msbm10 scaled\magstep1}
\newfont{\comp}{cmr12 scaled\magstep1}
\newfont{\compb}{cmr10 scaled\magstep2}
\newfont{\sbb}{cmssbx10 scaled\magstep3}
\newfont{\sbbb}{cmssbx10 scaled\magstep5}
\newfont{\sbs}{cmssbx10 scaled\magstep1}
\newtheorem{theorem}{Theorem}
\newtheorem{lemma}{Lemma}

\newtheorem{proposition}{Proposition}
\newtheorem{definition}[foo]{Definition}

\usepackage{hyperref}
\hypersetup{colorlinks = true, linkcolor=black, citecolor = {black}, urlcolor = black}

\title{On the order of Erd\H{o}s-Rogers functions}

\makeatletter
\renewcommand\@date{{%
  \vspace{-\baselineskip}%
  \large\centering
  \begin{tabular}{@{}c@{}}
    Dhruv Mubayi\thanks{Department of Mathematics, Statistics, and Computer Science, University of Illinois at Chicago, Chicago, USA.  Research partially supported by NSF grants
 1952767, 2153576 and a Simons fellowship. E-mail: mubayi@uic.edu.} \end{tabular}%
\quad \quad \quad  \begin{tabular}{@{}c@{}}
    Jacques Verstraete\thanks{Department of Mathematics, University of California, San Diego. Research supported by the National Science Foundation FRG Award DMS-1952786.
    E-mail: jacques@ucsd.edu} \\
  \end{tabular}

\bigskip
\bigskip

  \today
}}
\makeatother

\begin{document}

\setlength{\droptitle}{-5em}

\maketitle

\begin{abstract}
\noindent For an integer $n \geq 1$, the Erd\H{o}s-Rogers function $f_{s}(n)$ is the
maximum integer $m$ such that every $n$-vertex $K_{s+1}$-free graph has a $K_s$-free subgraph with
$m$ vertices. It is known that for all $s \geq 3$, $f_{s}(n) = \Omega(\sqrt{n\log n}/\log \log n)$ as $n \rightarrow \infty$.
In this paper, we show  that for all $s \geq 3$,
\begin{equation*}
f_{s}(n) = O(\sqrt{n}\, \log n).
\end{equation*}
 This improves previous bounds of order $\sqrt{n} (\log n)^{2(s + 1)^2}$ by Dudek, Retter and R\"{o}dl.
\end{abstract}

\section{Introduction}

 For an integer $s \geq 2$, the {\em Erd\H{o}s-Rogers function} $f_{s}(n)$ is the
maximum integer $m$ such that every $n$-vertex $K_{s + 1}$-free graph has a $K_s$-free subgraph with
$m$ vertices. The determination of $f_2(n)$ is almost equivalent to determining the triangle-complete graph Ramsey numbers, since $r(3,f_2(n)) \leq n < r(3,f_2(n) + 1)$.
These quantities $r(3,t)$ are known to within a constant factor~\cite{AKS,BK10,FGM,K,Sh}, and from this one deduces $f_2(n)$ has order of magnitude $\sqrt{n\log n}$ as
$n \rightarrow \infty$.  As observed by Dudek and the first author, the arguments for lower bounds for $f_2(n)$ generalize to $f_{s}(n)$ for $s \geq 3$. Shearer~\cite{Sh2} showed that any $n$-vertex $K_{s + 1}$-free graph of maximum degree $d$ has
an independent set of size $\Omega((n\log d)/(d\log\log d))$, and the neighborhood of a vertex of degree $d$ is a $K_s$-free induced subgraph. Therefore for all $s \geq 3$,
\begin{equation}
f_{s}(n) = \Omega\Bigl(\frac{\sqrt{n\log n}}{\log\log n}\Bigr).
\end{equation}
Wolfovits~\cite{W} proved $f_3(n) = O(\sqrt{n}(\log n)^{120})$, and it was shown by Dudek,  Retter, R\"{o}dl~\cite{DRR} that $f_3(n) = O(\sqrt{n} \,(\log n)^{32})$ and more generally that $f_{s}(n) = O(\sqrt{n}\,(\log n)^{2(s+1)^2})$. In this short paper, we significantly improve these bounds on the Erd\H{o}s-Rogers functions $f_s(n)$ as follows:

\begin{theorem}\label{thm:main}
For each fixed $s \geq 3$,
\begin{equation}
f_{s}(n) \; \; = \; \; O(\sqrt{n} \, \log n).\end{equation}
\end{theorem}

The proof of Theorem \ref{thm:main} involves a combination of the ideas of Wolfovits~\cite{W} and Dudek, Retter, R\"odl~\cite{DRR} with
the construction of Mattheus and the second author~\cite{MattV}, but does not make use of the method of containers as in~\cite{MattV} or~\cite{JS}. We did not expend too much effort in optimizing the implicit constant in the bound on $f_{s}(n)$ in
Theorem \ref{thm:main}; from the proof one may obtain $f_{s}(n) \leq 2^{100s} \sqrt{n} \log n$ for $n \geq 2$, which shows
$f_s(n) = n^{1/2 + o(1)}$ for $s = o(\log n)$.

\newpage

{\bf Notation.} For a graph $G$, we write $V(G)$ for the vertex set of $G$ and $E(G)$ for the edge set of $G$. For a set $X \subseteq V(G)$, let $G[X]$ denote the subgraph of $G$ {\em induced} by $X$,
namely the graph with vertex set $X$ and edge set $\{e \in E(G) : e \subseteq X\}$.

\section{Tools from probability}\label{prob}

We refer to the book by Alon and Spencer~\cite{AS} for a reference on probabilistic methods in combinatorics.
The first result we need is the {\em Chernoff Bound} (see Alon and Spencer~\cite{AS}) for concentration of binomial random variables.

\begin{proposition}\label{chernoff} {\bf (Chernoff Bound)}
Let $Z$ be a binomial random variable with mean $\mu$. Then for any real $\epsilon \in [0,1]$,
\begin{align}
  \label{eq:chernoff}
  \begin{split}
    \Pr(Z > (1 + \epsilon)\mu) &\leq \exp\Bigl(-\frac{\epsilon^2 \mu}{4}\Bigr) \quad \quad \mbox{\rm and }\\
    \Pr(Z < (1 - \epsilon)\mu) &\leq \exp\Bigl(-\frac{\epsilon^2 \mu}{2}\Bigr)
  \end{split}
\end{align}
\end{proposition}

\bigskip

The next proposition is derived from {\em Janson's inequality}~\cite{AS,J}. Let $\chi$ be a coloring of an $n$-element set $Y$ with $s$ colors, with color classes $Y_1,Y_2,\dots,Y_s$. Then the {\em random $s$-partite graph}
$G_{n,\rho}(\chi)$ samples edges independently with probability $\rho$ from the complete multipartite graph
with parts $Y_1,Y_2,\dots,Y_s$. The following technical proposition is derived in a standard way from Janson's inequality~\cite{AS,J}, and we give a proof in the appendix:

\begin{proposition} {\bf (via Janson's inequality)} \label{janson}
Let $s \geq 3$, $n \geq 2^{40s}$ and $\rho = (8s/n)^{2/s}$, and let $\chi$ be an $s$-coloring of an $n$-element set
whose color classes have size at least $n/2s$ each. Then
\begin{equation}\label{jbound}
 \Pr(K_s \not\subseteq G_{n,\rho}(\chi)) \leq \exp(-2^{2s-4}n).
\end{equation}
\end{proposition}

\bigskip

We shall finally require the {\em Lov\'{a}sz local lemma}~\cite{AS,L} in the following form. We write $\overline{A}$ for the complement of an event $A$ in a probability space.

\begin{proposition} {\bf (Lov\'{a}sz local lemma)} \label{local}
Let $A_1,A_2,\dots,A_n$ be events in some probability space and suppose that for some
set $J_i \subset \{1,2,\dots,n\}$, $A_i$ is mutually independent of
$\{A_j : j \not \in J_i \cup \{i\}\}$. If there exist real numbers
$\gamma_i \in [0,1)$ such that
\begin{equation}\label{lll}
\Pr(A_i) \leq \gamma_{i}
\prod_{j \in J_{i}} (1 - \gamma_{j}),
\end{equation}
then
\[\Pr\Bigl(\bigcap_{i=1}^n \overline{A_i}\Bigr) > 0.\]
\end{proposition}

\bigskip

\section{Hermitian unitals and O'Nan configurations}

The proof of Theorem \ref{thm:main} appeals to a construction from projective geometry. We briefly
describe the geometry here and the tools from probabilistic combinatorics that we use.
Further geometric background is given in Barwick and Ebert~\cite{BE}, Brouwer and van Maldeghem~\cite{BVM22} and Piper~\cite{P}.

\subsection{Hermitian unitals in brief}

For a partial linear space $\mathcal{H}$, we let $P(\mathcal{H})$ denote the set of points of $\mathcal{H}$ and $L(\mathcal{H})$ denote the set of lines of $\mathcal{H}$.
A {\em unital} in the projective plane $\mathrm{PG}(2,q^2)$ is a set $\mathcal{U}$ of $q^3 + 1$ points such that every line of $\mathrm{PG}(2,q^2)$ intersects $\mathcal{U}$ in $1$ or $q + 1$ points -- the latter are referred to as {\em secants}. A {\em classical} or {\em Hermitian unital} $\mathcal{H}_q$ is a partial
linear space described in homogeneous co-ordinates as the following set of one-dimensional subspaces of $\mathbb F_{q^2}^3$:
\begin{equation*}
P(\mathcal{H}_q) = \{\left<x,y,z\right> \subset \mathbb F_{q^2}^3 : x^{q + 1} + y^{q + 1} + z^{q + 1} = 0\}.
\end{equation*}
Here arithmetic is in the finite field $\mathbb F_{q^2}$, and $\left<x,y,z\right>$ is the one-dimensional subspace of $\mathbb F_{q^2}^3$ generated by $(x,y,z)$. Then
$L(\mathcal{H}_q)$ consists of the intersections of secant lines with $P(\mathcal{H}_q)$, so that there are $q^2(q^2 - q + 1)$ lines in $\mathcal{H}_q$,
each containing exactly $q + 1$ points of $\mathcal{H}_q$.

\subsection{O'Nan configurations and fans}

One of the remarkable features of the Hermitian unital is that it does not contain the
so-called {\em O’Nan configuration}, namely the configuration of
four lines and six points in the left figure below:

\begin{center}
\includegraphics[width=3.3in]{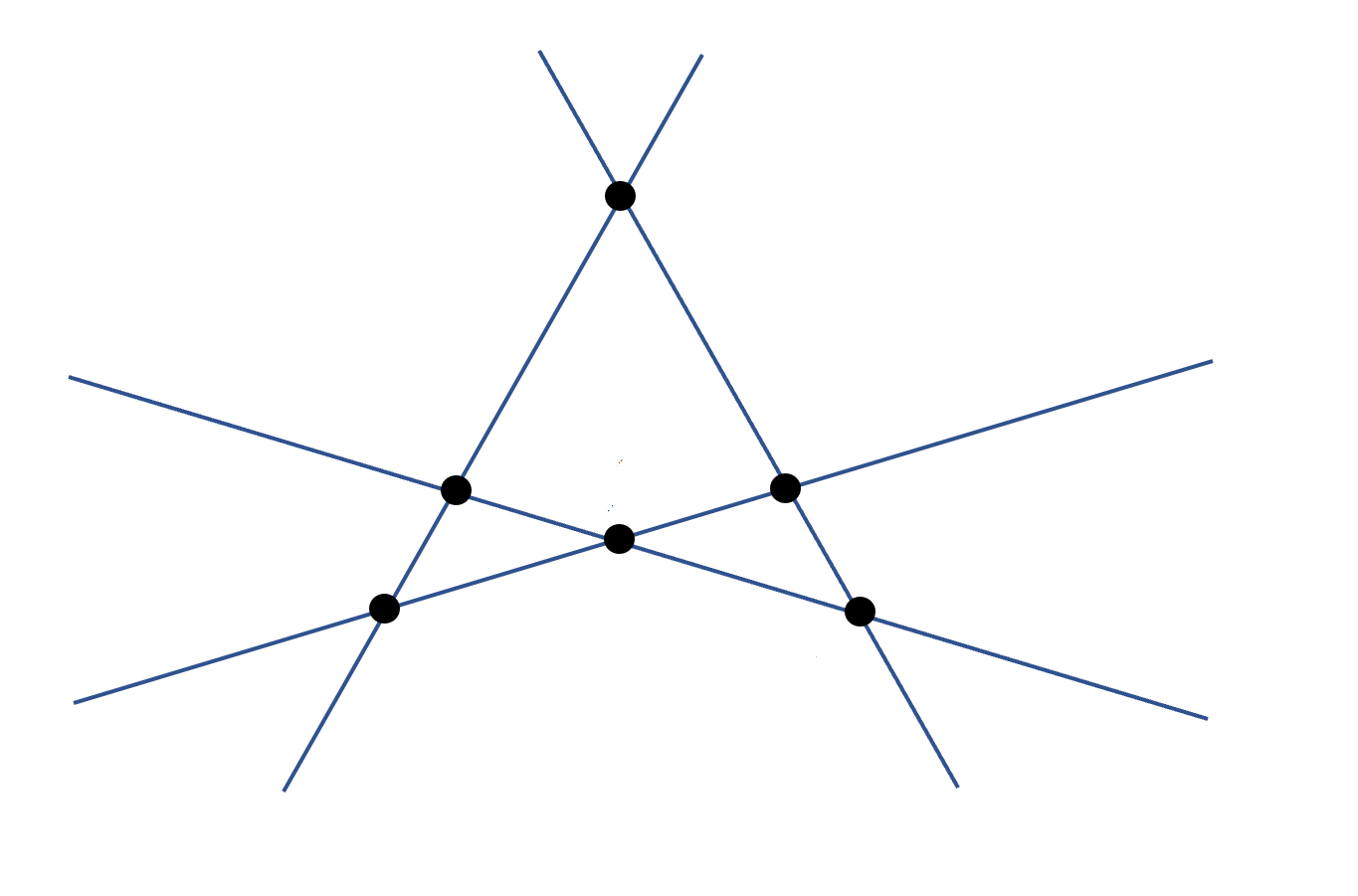}\includegraphics[width=3.3in]{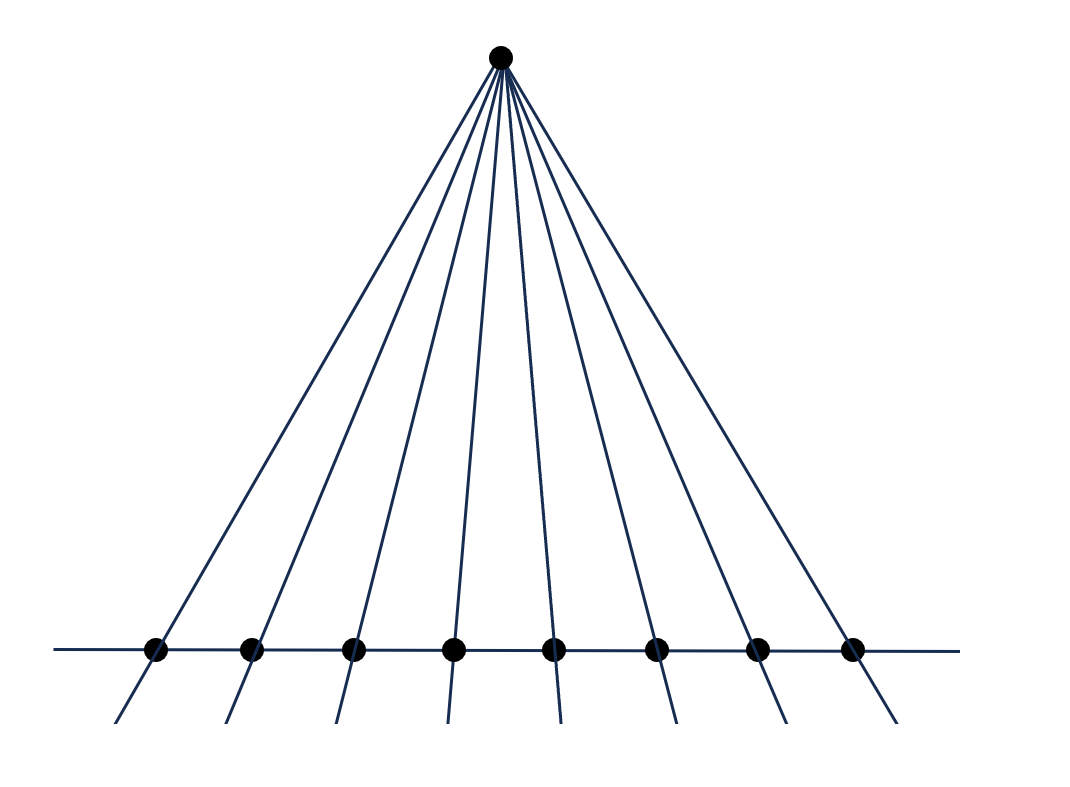}
\newline
O'Nan configuration and $s$-fan
\end{center}

\begin{definition} {\bf ($\boldsymbol{s}$-fan)}
For $s \geq 3$, an {\em $s$-fan} is a set of $s$ pairwise intersecting lines such that $s - 1$ of the lines are concurrent with a point whereas
the remaining line is not concurrent with that point. For $s \geq 4$, the unique point contained in $s - 1$ lines is the {\em point of concurrency} of the $s$-fan.
\end{definition}

An illustration of an $s$-fan is shown in the right figure above.  The fact that $\mathcal{H}_q$ does not contain the figure on the left was first proved by O'Nan~\cite{onan} (see Mattheus and the second author for a short linear-algebraic proof~\cite{MattV}). The following lemma is a straightforward consequence:

\begin{lemma}\label{no-onan}
If $s \geq 3$ lines in $\mathcal{H}_q$ pairwise cross, then they are concurrent with some point of $\mathcal{H}_q$ or they form an $s$-fan.
\end{lemma}

\begin{proof}
Let $K \subseteq L(\mathcal{H}_q)$ be a configuration of $s$ pairwise crossing lines. In the case $s = 3$, if the lines are not all concurrent with some point, then
they form a triangle, with is an $s$-fan. The case $s = 4$ follows from the fact that $\mathcal{H}_q$ contains no O'Nan configuration. If $s > 4$, then at least three of the lines $\ell_1,\ell_2,\ell_3 \in K$ are concurrent with some point $p \in \mathcal{H}_q$. Suppose for a contradiction that two of the lines $\ell_4,\ell_5 \in K$ are not concurrent with $p$.
Then three of $\ell_1,\ell_2,\ell_4,\ell_5$ are concurrent with some $p' \in P(\mathcal{H}_q) \backslash \{p\}$.
This implies $\ell_1,\ell_2,\ell_3,\ell_4,\ell_5$ are concurrent with $p'$, contradicting that $\ell_4,\ell_5$ are concurrent with $p$.
\end{proof}

\section{Random sampling}

To prove Theorem \ref{thm:main}, we require for $s \geq 3$ an $n$-vertex $K_{s+1}$-free graph $H$ such that every induced subgraph of $H$ with subtantially more than about $\sqrt{n}\log n$ vertices contains a copy of $K_s$. The overview of the construction of $H$ is as follows, where asymptotic notation is with respect to a growing prime power $q$.

\bigskip

First we randomly sample points from the Hermitian unital $\mathcal{H}_q$ with probability roughly $\Theta(\log q/(q + 1))$ for to obtain a partial linear space $\mathcal{H}$. We show in Section \ref{randompoints} via the Chernoff bound, Proposition~\ref{chernoff}, that with positive probability, $|P(\mathcal{H})| = \Theta(q^2\log q)$, $|L(\mathcal{H})| = q^2(q^2 - q + 1)$, each line in $L(\mathcal{H})$ has size $\Theta(\log q)$, and the number of $(s + 1)$-fans containing any given pair of intersecting lines in $\mathcal{H}$ is $\Theta(\log q)^s$.

\bigskip

We define the {\em intersection graph} $G$ whose vertex set is $L(\mathcal{H})$ and whose edges are pairs of intersecting lines of $\mathcal{H}$. The graph $G$ has $q^2(q^2 - q + 1)$ vertices and every edge of $G$ is contained in $\Theta(\log q)^s$ copies of $K_{s+1}$ in $G$. Most importantly,

\begin{equation}\tag{A}
\parbox{4.3in}{\sl the vertex set of each $K_{s+1} \subseteq G$ is either an $(s + 1)$-fan in $\mathcal{H}$ or is a set of lines of $\mathcal{H}$ all concurrent with some point in $\mathcal{H}$. }
\end{equation}

\medskip

We eliminate the copies of $K_{s+1} \subseteq G$ corresponding to $s + 1$ lines of $\mathcal{H}$ concurrent with a point $p$ of $\mathcal{H}$ by randomly $s$-coloring the set $\mathcal{H}_p$ of lines concurrent with $p$, independently over different points of $\mathcal{H}$, and retaining only those edges $\{\ell,\ell'\}$ of $G$ such that $\ell$ and $\ell'$ have different colors in the
coloring of $\mathcal{H}_p$ where $p = \ell \cap \ell'$. By (A), this graph $G_{\chi}$ has the following property:

\begin{equation}\tag{B}
\parbox{3.7in}{\sl the vertex set of each $K_{s+1} \subseteq G_{\chi}$ is an $(s + 1)$-fan.}
\end{equation}

We eliminate the copies of $K_{s + 1}$ in $G$ which correspond to $(s + 1)$-fans in $\mathcal{H}$ by randomly sampling edges of $G$ independently with suitable probability $\rho = \Theta(\log q)^{-2/s}$ to obtain a random graph $G_{\rho} \subseteq G$.
Let $H$ be the intersection of $G_{\rho}$ and $G_{\chi}$, that is, $V(H) = V(G)$ and $E(H) = E(G_{\rho}) \cap E(G_{\chi})$.
We apply the Lov\'{a}sz local lemma (Proposition~\ref{local}) in Section \ref{theproof} to show that with positive probability,
the graph $H$ is $K_{s + 1}$-free and, for a constant $C$ depending only on $s$, every set of $Cq^2\log q$ vertices of $H$ induces a subgraph of $H$ containing $K_s$. Then $|V(H)| = n = q^2(q^2 - q + 1)$ and so $f_s(n) \leq Cq^2\log q$. The distribution of primes $q$ allows us to deduce $f_s(n) = O(\sqrt{n}\log n)$ for all $n$.

\subsection{Randomly sampling points}\label{randompoints}

The first step in our construction is to randomly sample points from $\mathcal{H}_q$ to obtain a partial linear space $\mathcal{H}$ with the properties listed below, which essentially uses only the Chernoff Bound, Proposition \ref{chernoff}.

\begin{lemma}\label{sampleH}
For any integers $s \geq 3$ and $a \geq 128$ and any prime power $q \geq a\log q$, there exists a partial linear space $\mathcal{H} = \mathcal{H}_{a,q,s}$ such that
\begin{center}
\begin{tabular}{lp{5.9in}}
{\rm (i)} & The line set $L$ of $\mathcal{H}$ has size $q^2(q^2 - q + 1)$. \\
{\rm (ii)} & The point set $P$ of $\mathcal{H}$ has size at most $2a q^2\log q$ and at least $a (q^2\log q)/2$. \\
{\rm (iii)} & Each line of $\mathcal{H}$ has at least $(a\log q)/2$ points. \\
{\rm (iv)} & The number of $(s + 1)$-fans in $\mathcal{H}$ containing a pair of lines in $L$ is at most $k = (2a\log q)^s$. \\
{\rm (v)} & Every $s + 1$ pairwise intersecting lines in $\mathcal{H}$ is concurrent with a point or is an $(s + 1)$-fan.
\end{tabular}
\end{center}
\end{lemma}

\begin{proof}
Sample points of $\mathcal{H}_q$ randomly and independently with probability $(a\log q)/(q + 1)$ and let $\mathcal{H}$ be the partial linear space
whose point set is the set $P$ of sampled points and whose line set is $L = \{\ell \cap P : \ell \in L(\mathcal{H}_q)\}$.
Since we are sampling points, $|L(\mathcal{H}_q)| = |L|$ and so (i) holds, and (v) follows from Lemma \ref{no-onan}.

\bigskip

By the Chernoff bound (Proposition \ref{chernoff}) with $\epsilon = 1/2$, the probability that $|P|$ is larger than  $2a q^2 \log q$ or less than $(aq^2\log q)/2$ is at most
$2\exp(-(aq^2\log q)/16) < 1/3$, and the number of sampled points on a given line $\ell \in L(\mathcal{H}_q)$ is at most $(a\log q)/2$ with probability at most $\exp(-(a\log q)/8) = q^{-a/8} < q^{-8}$. Since $|L(\mathcal{H}_q)| \leq q^4 < q^8/3$, the probability that (iii) fails is less than 1/3.

\bigskip

If the number of $(s + 1)$-fans on some pair of lines $\ell$ and $\ell'$ is more than $k$, then we claim some
line had more than $2a\log q$ points sampled on it. This happens
with probability at most $\exp(-(a\log q)/4) < q^{-8}$ by the Chernoff Bound (Proposition \ref{chernoff}) with $\epsilon = 1$. Since the number of pairs of intersecting lines in $\mathcal{H}_q$ is
at most $(q^3 + 1) \cdot {q^2 \choose 2} \leq q^8/3$, (iv) then fails with probability less than 1/3. We then conclude (ii) -- (iv) hold simultaneously with positive probability.

\bigskip

To prove the claim, fix distinct lines $\ell$ and $\ell'$ in $\mathcal{H}_q$. For each $(s + 1)$-fan in $\mathcal{H}$ containing $\ell$ and $\ell'$,  either $\{p\} = \ell \cap \ell'$ is the point of concurrency or some other point of $\ell \cup \ell' \backslash \{p\}$ is the point of concurrency.
If $p$ is the point of concurrency, then we must pick a point on $\ell$ and a point in $\ell'$ to define a line $\ell''$ through the remaining points of the $(s + 1)$-fan.
There are at most $(2a\log q)^{2}$ choices for those two points, and then since $\ell''$ has at most $2a\log q$ points sampled, there are at most $(2a\log q)^{s-2}$ choices for the remaining $s - 2$ points of the $(s + 1)$-fan. If $p$ is not the point of concurrency, then the point of concurrency is picked from $\ell \cup \ell' \backslash \{p\}$ in at most $2(2a\log q)$ ways. Then we must pick the remaining $s - 1$ points of the $(s + 1)$-fan from $\ell$ or from $\ell'$, in at most $2(2a\log q)^{s-1}$ ways.
So the number of $(s + 1)$-fans containing $\ell$ and $\ell'$ is at most $4(2a\log q)^{s} \leq (2a\log q)^{s} = k$.
\end{proof}

\bigskip

For a point $p \in P$ and a set $X \subseteq L$, let $X_p$ denote the set of lines $X$ concurrent with the point $p$.
It is convenient for a positive integer $b$ and $X \subseteq L$ to define
\[ P_X = P_{X,b} = \{p \in P : |X_p| \geq b\}.\]

\bigskip

\begin{lemma}
Let $b \geq 1$, $a \geq 128$ and $q \geq a\log q$. Then for any $X \subseteq L$,
\begin{equation}\label{b-bound}
\sum_{p \in P_X} |X_p| > \frac{1}{2}(a\log q) \cdot |X| - 2abq^2\log q.
\end{equation}
\end{lemma}

\begin{proof} As $|P| \leq 2aq^2\log q$ from Lemma \ref{sampleH}.ii,
\[ \sum_{p \in P \backslash P_X} |X_p| < b|P| \leq 2abq^2\log q.\]
Each line in $X$ is incident with at least $(a\log q)/2$ points by Lemma \ref{sampleH}.iii. Therefore
\[ \sum_{p \in P} |X_p| = \sum_{\ell \in X} |\ell \cap P| \geq \frac{1}{2}(a\log q) \cdot |X|.\]
Subtracting the first inequality from the second gives the lemma.
\end{proof}

\subsection{Random sampling of pairs of lines}\label{randomlines}

We use the partial linear space $\mathcal{H} = \mathcal{H}_{a,q,s}$ in Lemma \ref{sampleH} for $a \geq 128$ and $q \geq a\log q$ to construct the {\em intersection graph} $G = G_{a,q,s}$ of lines in $\mathcal{H}$
as follows: the vertex set of $G$ is $L$ and the edge set is $\{\{\ell,\ell'\} : \ell \cap \ell' \neq \emptyset\}$. We use the words {\em line} and {\em vertex}
interchangeably to refer to the vertices of $G$. From Lemma \ref{sampleH}, for each prime power $q \geq a\log q$, the graph $G$ has $|L| = q^2(q^2 - q + 1)$ vertices, and most importantly, for $s \geq 3$,

\begin{equation}\tag{A}
\parbox{4.3in}{\sl the vertex set of each $K_{s+1} \subseteq G$ is either an $(s + 1)$-fan in $\mathcal{H}$ or is a set of lines of $\mathcal{H}$ all concurrent with some point in $\mathcal{H}$. }
\end{equation}

\medskip

 For each point $p \in P$, the set of lines
concurrent with $p$ induces a clique of size $q^2$ in $G$, and any two of these cliques are edge-disjoint. Fixing a set $X \subseteq L$ and a point $p \in P$, recall $X_p \subseteq X$ is the set of lines in $X$ concurrent with $p$, and in particular for $X_p \neq \emptyset$ the subgraphs $G[X_p]$ are edge-disjoint cliques for $p \in P$.

\bigskip

Then let $G_{\chi} \subseteq G$ be obtained from $G$ by taking independently for each
$p \in P$ a random vertex-coloring $\chi_p$ of all the lines in $G$ concurrent with $p$
and removing all edges of $G$ whose ends have the same color. This removes from $G$ all
copies of $K_{s + 1}$ induced by $s + 1$ lines concurrent with a single point. We conclude by (A), for $s \geq 3$:

\begin{equation}\tag{B}
\parbox{3.7in}{\sl the vertex set of each $K_{s+1} \subseteq G_{\chi}$ is an $(s + 1)$-fan.}
\end{equation}

\medskip

We now define the random graph $G_{\rho} \subseteq G$. Let $b \geq 1$ and  $\rho \in [0,1]$ satisfy $b \geq 2^{40s}$ and $\rho = (8s/b)^{2/s}$, and define $G_{\rho}$ to be the random graph obtained by sampling edges of $G$ independently with probability $\rho$.
We let $H$ be the intersection of $G_{\rho}$ and $G_{\chi}$, namely, the graph with vertex set $V(G)$ and edge set $E(G_{\rho}) \cap E(G_{\chi})$. Our next task is to consider copies of $K_s \subseteq H$ whose vertices are contained in some set of lines concurrent with a single point. Unsurprisingly, this makes strong use of Janson's inequality for the probability that a random $s$-partite graph is $K_s$-free, in the form of Proposition \ref{janson}.

\bigskip

\subsection{The event $A_X$}\label{axsection}

In this section, for each set $X$ of lines, we define an event $A_X$ whose non-occurrance implies $H[X]$ contains a copy of $K_s$, and we find an upper bound on $\Pr(A_X)$.
For a set $X \subseteq L$ and a point $p \in P_X$, fix a family $\Pi_p(X) = \Pi_p$ of $r_p(X) = \lfloor |X_p|/b\rfloor$ disjoint subsets of $X_p$ of size $b$ each. Then $X_p$ is {\em bad} if for $Y \in \Pi_p$, none of $H[Y]$ contains a $K_{s}$, and we let $A_{X,p}$ be the event that $X_p$ is bad.
We say $X$ is {\em bad} if all $X_p : p \in P_X$ are bad. In other words, if $A_X$ is the event that $X$ is bad, then
\[ A_X = \bigcap_{p \in P_X} A_{X,p}.\]
If $A_X$ does not occur, then by definition $H[X]$ contains a copy of $K_s$.

\begin{lemma}\label{mainlemma}
Let $s \geq 3$, $b \geq 2^{40s}$ and $\rho \in [0,1]$ satisfy $\rho = (8s/b)^{2/s}$. Then for any $X \subseteq L$,
\begin{equation}\label{notgood3}
\Pr(A_X) \leq \exp\Bigl(-\frac{1}{32s}\sum_{p \in P_X} |X_p| \Bigr).
\end{equation}
\end{lemma}

\begin{proof}
Since the colorings $\chi_p$ are independent over $p \in P_X$, the events $A_{X,p}$ are independent over $p \in P_X$.
For $Y \in \Pi_p = \Pi_{p}(X)$, the events $A_Y$ that $H[Y]$ does not contain $K_{s}$ are independent. By the Chernoff bound (Proposition \ref{chernoff}), the probability that
$\chi_p$ assigns some color fewer than $b/2s$ times to vertices of $Y$ is at most
$\exp(-b/8s)$. Fix a coloring $\chi$ of $Y$ where every color appears at least $b/2s$  times. Then the graph $H[Y]$ is a random $s$-partite
graph which we denoted in Section \ref{prob} by $G_{b,\rho}(\chi)$. By Proposition \ref{janson}, for any such coloring $\chi$,
\[ \Pr(K_s \not\subseteq G_{b,\rho}(\chi)) \leq \exp(-2^{2s-4}b).\]
As there are at most $s^b$ choices of $\chi$, the union bound over $s$-colorings $\chi$ gives
\[ \Pr(A_Y) \leq \exp\Bigl(-\frac{b}{8s}\Bigr) + s^b \exp(-2^{2s-4}b).\]
Since $s \geq 3$, $2^{2s - 4}b \geq 2b\log s \geq 2b$ and therefore
\[ \Pr(A_Y) \leq \exp\Bigl(-\frac{b}{8s}\Bigr) + \exp(-b) \leq 2\exp\Bigl(-\frac{b}{8s}\Bigr) \leq \exp\Bigl(-\frac{b}{16s}\Bigr).\]
Here we used $b \geq 16s$. Since $|\Pi_p| = r_p(X)$, we obtain
\begin{eqnarray*}
\Pr(A_X) &=& \prod_{p \in P_X} \Pr(A_{X,p}) \\
&=& \prod_{p \in P_X} \prod_{Y \in \Pi_p} \Pr(A_Y) \; \; \leq \; \; \prod_{p \in P_X} \exp\Bigl(-\frac{b}{16s} \cdot r_p(X)\Bigr).
\end{eqnarray*}
Recall $|X_p| \geq b$ for $p \in P_X$, so $b \cdot r_p(X) \geq |X_p|/2$, and therefore
\[ \Pr(A_X) \leq \prod_{p \in P_X} \exp\Bigl(-\frac{|X_p|}{32s}\Bigr) = \exp\Bigl(-\frac{1}{32s} \sum_{p \in P_X} |X_p|\Bigr).
\]
This proves the lemma.
\end{proof}

\section{Proof of Theorem \ref{thm:main}}\label{theproof}

To prove Theorem \ref{thm:main}, for each $s \geq 3$ let $G = G_{a,q,s}$ be the intersection graph defined in Section \ref{randomlines}, where $a = 2^{10} s$ and $q$ is a prime power satisfying $q \geq a\log q$.
Let $H \subseteq G$ denote the random graph defined in Section \ref{randomlines}, which is the intersection of the two
random graphs $G_{\chi}$ and $G_{\rho}$, with parameters
\[ b = 2^{40s} \cdot 2a\log q \qquad \mbox{ and } \qquad \rho = \Bigl(\frac{8s}{b}\Bigr)^{\frac{2}{s}}.\]
For convenience we omit rounding and assume $b$ is an integer. Let $\mathcal{K}$ be the family of sets of $s + 1$ lines $K \subseteq L$ forming an $(s + 1)$-fan in $G$, so $G[K]$ is a complete graph of order $s + 1$, and let $A_K$ be the event $G[K] \subseteq G_{\rho}$.
 Let $\mathcal{X} = \{X \subseteq L : |X| = 8bq^2\}$ and $A_X$ be the event $X$ is bad, as in Section \ref{axsection}. Due to (B),

\begin{equation}\tag{C}
\parbox{3.7in}{\sl if none of the events $A_K : K \in \mathcal{K}$ or $A_X : X \in \mathcal{X}$ occur, then $H$ is $K_{s+1}$-free
and $K_s \subseteq H[X]$ for all $X \in \mathcal{X}$.}
\end{equation}

\medskip

Specifically, (C) implies every set of $8bq^2$ vertices of $H$ induces a subgraph containing $K_s$.
This shows for any prime power $q \geq a\log q$,
\[ f_s(q^2(q^2 - q + 1)) \leq 8bq^2. \]
By Bertrand's postulate, there exists a prime between any positive integer and its double, so letting $q \geq a\log q$ be a
prime between $n^{1/4}$ and $2n^{1/4}$, we find for $s \geq 3$ and $n \geq 2$,
\begin{eqnarray*}
 f_s(n) &\leq& 8bq^2 \leq 8 \cdot 2^{40s} \cdot 2^{11}s \cdot q^2 \cdot \log q \\
 &\leq& 2^{100s} \cdot \sqrt{n} \log n.
\end{eqnarray*}
It remains to prove (C) holds with positive probability, via the local lemma (Proposition \ref{local}).

\bigskip

{\bf Dependencies.} For the dependencies between the events $A_K : K \in \mathcal{K}$ and $A_X : X \in \mathcal{X}$, we note $A_X$ is determined by the following set of edges of $G$:
\[ \hat{E}(X) = \bigcup_{p \in P_X} \bigcup_{Y \in \Pi_p} E(G[Y]). \]
Since $|Y| = b$ for each $Y \in \Pi_p$, and $r = r_p(X) = |\Pi_p| = \lfloor |X_p|/b\rfloor$,
\begin{eqnarray}
|\hat{E}(X)| &=& \sum_{p \in P_X} \sum_{i = 1}^{r} {b \choose 2} \notag \\
&=&  \sum_{p \in P_X} \Big\lfloor \frac{|X_p|}{b}\Big\rfloor {b \choose 2} \; \; \leq \; \; \frac{1}{2} b \sum_{p \in P_X} |X_p|.
\label{ebound}
\end{eqnarray}
For convenience, let $\hat{E}(K)$ also denote $E(G[K])$ if $K \in \mathcal{K}$.
It is important here that since $G_{\rho}$ samples edges of $G$ independently with probability
$\rho$, for any $\mathcal{J} \subseteq \mathcal{K}$ and $K \in \mathcal{K}$, we observe

\begin{equation*}
\parbox{4.3in}{\sl the event $A_K$ is mutually independent with $\{A_{K'} : K' \in \mathcal{J}\}$ if
\[ \bigcup_{K' \in \mathcal{J}} E(K') \cap E(K) = \emptyset.\]}
\end{equation*}

Let $k = (2a\log q)^s$. By Lemma \ref{sampleH}.iv, each edge of $G$ is contained in at most $k$ cliques $K \in \mathcal{K}$.
Fixing any $K \in \mathcal{K}$, there are at most
\begin{equation}\label{kappa}
 \kappa = {s+1 \choose 2} \cdot k \leq bk
 \end{equation}
choices of $K' \in \mathcal{K}$ such that $\hat{E}(K') \cap \hat{E}(K) \neq \emptyset$, and any other set of events $A_{K'}$ are
mutually independent with $A_K$ by the observation above. We assume
all $A_X$ are mutually dependent with any given event $A_K$.

\bigskip

Since each edge of $G$ is in at most $k$ cliques in $\mathcal{K}$,
by (\ref{ebound}), for each $X \in \mathcal{X}$ there are at most
\begin{eqnarray}
\lambda &=& k \cdot |\hat{E}(X)| \notag \\
&\leq& \frac{1}{2}bk \cdot \sum_{p \in P_X} |X_p| \label{lambda}
\end{eqnarray}
choices of $K \in \mathcal{K}$ such that $\hat{E}(K) \cap \hat{E}(X) \neq \emptyset$, and any set of other events
$A_K$ are mutually independent with $A_X$. We assume all $A_{X'}$ are mutually dependent with any given event $A_X$.

\bigskip

{\bf Local lemma inequalities.} Let $N = |\mathcal{X}|$. The local lemma (Proposition \ref{local}) implies the probability that (C) holds is positive if (\ref{lll}) holds, i.e. there are reals $\gamma,\delta \in [0,1)$ such that for all $K \in \mathcal{K}$ and all $X \in \mathcal{X}$,
\begin{equation*}
\Pr(A_K) \leq \gamma (1 - \gamma)^{\kappa} (1 - \delta)^{N} \quad \quad \mbox{ and } \quad \quad
\Pr(A_X) \leq \delta (1 - \gamma)^{\lambda} (1 - \delta)^{N}.
\end{equation*}
We claim that these inequalities are satisfied if we select $\delta = 1/(N + 1)$ and $\gamma = 1/32sbk$.

\bigskip

{\bf First inequality.} We have $\Pr(A_K) = \rho^{{s+1 \choose 2}}$ and $(1 - \delta)^N \geq 1/e$. By (\ref{kappa}),
$(1 - \gamma)^{\kappa} \geq 1 - \kappa \gamma > 1/2$, so it is sufficient to show
$2e \Pr(A_K)/\gamma \leq 1$ for the first inequality to hold.  Since $\rho = (8s/b)^{2/s}$,
\[ \rho^{{s + 1 \choose 2}} = \Bigl(\frac{8s}{b}\Bigr)^{s+1}.\]
Since $b = 2^{40s} \cdot 2a\log q = 2^{40s} k^{1/s}$, and $s \geq 3$,
\begin{eqnarray*}
\frac{2e}{\gamma} \Pr(A_K) &=& 64esbk \cdot \rho^{{s+1 \choose 2}} \\
&\leq& \frac{64esk (8s)^{s+1}}{b^{s}} \\
&=& \frac{64esk (8s)^{s + 1}}{2^{40s^2}k} \\
&<& \frac{64esk (8s)^{s + 1}}{64^{s^2}k} \\
&=&  es \cdot \Bigl(\frac{8s}{64^{s - 1}}\Bigr)^{s+1} \\
&<& \frac{es}{8^{s + 1}} \; \; < \; \; 1.
\end{eqnarray*}
Here we used $64^{s-1} > 64s$ and $es < 8^{s+1}$ for $s \geq 3$.
This verifies the first inequality.

\bigskip

{\bf Second inequality.} For the second inequality, we use $1 - \gamma \geq \exp(-2\gamma)$, which is valid since
$\gamma \leq 1/2$. Recalling $(1 - \delta)^N \geq 1/e$, it is enough to show
\[ e \cdot \Pr(A_X) \leq \exp(-\log(N + 1) - 2\gamma\lambda).\]
By Lemma \ref{mainlemma},
\[ \Pr(A_X) \leq \exp\Bigl(-\frac{1}{32s}\sum_{p \in P_X} |X_p|\Bigr).\]
Using $|L| = |V(G)| = q^2(q^2 + q + 1)$,
\[ \log(N + 1) = \log \Bigl[{q^2(q^2 - q + 1) \choose 8bq^2} + 1\Bigr] \leq \log {q^4 \choose 8bq^2} - 1 \leq  32bq^2\log q - 1.\]
Due to the upper bound on $\lambda$ given by (\ref{lambda}), it is enough to show
\[ \exp\Bigl(-\frac{1}{32s}\sum_{p \in P_X} |X_p|\Bigr) \leq \exp\Bigl(-32 bq^2\log q + \frac{1}{64s} \sum_{p \in P_X} |X_p|\Bigr).\]
Therefore we require
\[ \exp\Bigl(-\frac{1}{64s}\sum_{p \in P_X} |X_p|\Bigr) \leq \exp(-32 bq^2\log q).\]
Applying (\ref{b-bound}) and using $a = 2^{10}s$, we find
\begin{eqnarray*}
\frac{1}{64s} \sum_{p \in P_X} |X_p| &\geq& \frac{1}{64s} \Bigl(\frac{1}{2}(a\log q) \cdot |X| - 2abq^2\log q\Bigr) \\
&=& \frac{1}{64s} \Bigl(\frac{1}{2}(a\log q) \cdot 8bq^2 - 2abq^2\log q\Bigr)\\ \\
&=& \frac{1}{64s} \cdot 2abq^2\log q \; \; = \; \; 32bq^2\log q.
\end{eqnarray*}
This proves the second inequality. \qed

\bigskip

\section*{Concluding remarks}

$\bullet$ In this paper, we proved $f_s(n) = O(\sqrt{n}\log n)$ by suitable random sampling of points and lines from Hermitian unitals.
A part of the proof essentially involves the union bound over all sets $X$ of lines of size $8bq^2$ -- we
implicitly assumed in our application of the local lemma in Section \ref{theproof} that the events $A_X$ depend on all other such events.  We first sampled points randomly from the Hermitian unital with probability of order $(\log q)/q$. If we sampled points with a lower probability $o(\log q)/q$, then the union bound no longer works: for large $q$ there are $N \geq \exp(bq^2\log q)$
sets of size $8bq^2$ to consider -- see Section \ref{theproof} -- whereas if all the sets $X_p : p \in P_X$ have size roughly $b$, then the probability that
every $X_p$ is $(s - 1)$-colored in a random $s$-coloring is $\exp(-o(bq^2\log q))$. We believe it may be possible using
randomized greedy algorithms akin to the R\"{o}dl semirandom method or more recent iterative absorption methods to circumvent this
issue and obtain a bounds $f_s(n) = o(\sqrt{n}\log n)$, but did not investigate this technical direction.

\bigskip

$\bullet$ It is plausible that when $s \geq 3$, $f_s(n) = \Omega(\sqrt{n}(\log n)^{\alpha})$ for some $\alpha > 1/2$.
The current lower bound of order $n\sqrt{\log n}/\log\log n$ is achieved by finding an induced $K_s$-free subgraph of size $d$ (the neighborhood of a vertex of degree $d$) or an independent set of size $n(\log d)/d(\log\log d)$ in a $K_{s+1}$-free graph with maximum degree $d$. The latter seems potentially wasteful, in that we could perhaps find for $s \geq 3$ an induced $K_s$-free subgraph of size
much larger than $n(\log d)/d(\log\log d)$ in a $K_{s+1}$-free graph. In addition, it is believable that the $\log\log d$ term is superfluous; perhaps
every $K_{s+1}$-free graph with maximum degree $d$ has an independent set of order at least $n(\log d)/d$.

\bigskip

\section{Appendix : Proof of Proposition \ref{janson}}

To prove Proposition \ref{janson}, we require some notation and preliminaries. Recall for $s \geq 3$, $n \geq 1$, $\rho \in [0,1]$ and an $s$-coloring
$\chi$ of an $n$-element set with color classes $Y_1,Y_2,\dots,Y_s$, $G_{n,\rho}(\chi)$
is obtained by independently and randomly sampling edges of the complete $s$-partite graph with parts $Y_1,Y_2,\dots,Y_s$ with
probability $\rho$. The expected number of  $K_s \subseteq G_{n,\rho}(\chi)$ is precisely
\begin{equation}\label{mudef}
\mu(\chi) = \rho^{{s \choose 2}} \prod_{i = 1}^s |Y_i|
\end{equation}
and the variance is defined by
\begin{equation}\label{tridef}
\triangle(\chi) = \sum_{S \subset [s]} \rho^{2{s \choose 2} - {|S| \choose 2}}
\prod_{i \in S} |Y_i| \prod_{j \not \in S} |Y_j|(|Y_j| - 1) = \sum_{S \subset [s]} \rho^{2{s \choose 2} - {|S| \choose 2}}
\prod_{i = 1}^n |Y_i| \prod_{j \not \in S} (|Y_j| - 1),
\end{equation}
where the sum is over sets $S$ with $2 \leq |S| \leq s - 1$. From Janson's inequality~\cite{AS,J} one obtains for any $\rho \in [0,1]$ and $n \geq 1$:
\begin{equation}\label{j}
 \Pr(K_s \not\subseteq G_{n,\rho}(\chi)) \leq \exp\Bigl(-\frac{1}{2}\mu(\chi)\Bigr).
 \end{equation}
provided $\triangle(\chi) \leq \mu(\chi)$.

\bigskip

{\bf Proof of Proposition \ref{janson}.} For Proposition \ref{janson}, $\chi$ is an $s$-coloring
with color classes $Y_1,Y_2,\dots,Y_s$ satisfying $|Y_i| \geq n/2s$ for all $i \in [s]$. The
product in (\ref{mudef}) is minimized when
$|Y_i| = n/2s$ for all but one value of $i \in [s]$, and the remaining color class has
size $n - (s - 1)n/2s > n/2$. Therefore
\begin{eqnarray*}
\mu(\chi) &>& \rho^{{s \choose 2}} \Bigl(\frac{n}{2s}\Bigr)^{s - 1} \cdot \frac{n}{2} \\
&=& \Bigl(\frac{8s}{n}\Bigr)^{s - 1} \Bigl(\frac{n}{2s}\Bigr)^{s - 1} \cdot \frac{n}{2} \; \; =\; \; 2^{2s - 3} \cdot n.
\end{eqnarray*}
So if $\triangle(\chi) \leq \mu(\chi)$ when $\rho = (8s/n)^{2/s}$ and $n \geq 2^{40s}$,
then Proposition \ref{janson} follows from (\ref{j}):
\[  \Pr(K_s \not\subseteq G_{n,\rho}(\chi)) \leq  \exp\Bigl(-\frac{1}{2}\mu(\chi)\Bigr) \leq \exp(-2^{2s-4}n).\]
It remains to prove $\triangle(\chi) \leq \mu(\chi)$ when $\rho = (8s/n)^{2/s}$ and $n \geq 2^{40s}$.

\bigskip

By (\ref{mudef}) and (\ref{tridef}),
\[ \frac{\triangle(\chi)}{\mu(\chi)} = \sum_{S \subset [s]} \rho^{{s \choose 2} - {|S| \choose 2}} \prod_{j \not \in S} (|Y_j| - 1).\]
The sum is over subsets $S$ of $[s]$ where $2 \leq |S| \leq s - 1$. By the inequality of geometric and arithmetic means,
\[ \prod_{j \not \in S} (|Y_j| - 1) \leq \prod_{j \not \in S} |Y_j| \leq \Bigl(\frac{1}{s - |S|} \sum_{j \not \in S} |Y_j|\Bigr)^{s - |S|}
\leq \Bigl(\frac{n}{s - |S|}\Bigr)^{s - |S|}.\]
We conclude
\[  \frac{\triangle(\chi)}{\mu(\chi)} \leq \sum_{i = 2}^{s - 1} \rho^{{s \choose 2} - {i \choose 2}} \Bigl(\frac{n}{s - i}\Bigr)^{s - i} {s \choose i}.\]
By definition of $\rho$,
\[ \rho^{{s \choose 2} - {i \choose 2}} = \Bigl(\frac{8s}{n}\Bigr)^{\frac{(s - i)(s + i - 1)}{s}}.\]
Therefore
\begin{eqnarray*}
\frac{\triangle(\chi)}{\mu(\chi)} &\leq& \sum_{i = 2}^{s - 1} \Bigl(\frac{8s}{n}\Bigr)^{\frac{(s - i)(s + i - 1)}{s}}
\Bigl(\frac{n}{s - i}\Bigr)^{s - i} {s \choose i} \; \; = \; \;  \sum_{i = 2}^{s-1} \Bigl(\frac{(8s)^{\frac{s+i-1}{s}}}{(s-i)n^{\frac{i-1}{s}}}\Bigr)^{s - i} {s \choose i}.
\end{eqnarray*}
We break the sum into two pieces. First, for $2 \leq i \leq \lfloor s/\log(8s)\rfloor \leq s/2$,
\[ (s - i)n^{\frac{i-1}{s}} \geq \frac{s}{2} \cdot n^{\frac{1}{s}} \geq 2^{9}s\]
since $n \geq 2^{40s} \geq 2^{10s}$. Therefore each term in the sum is at most
\[ \Bigl(\frac{8s \cdot (8s)^{\frac{1}{\log(8s)}}}{2^9s}\Bigr)^{s - i} {s \choose i}
\leq \Bigl(\frac{8es}{2^9s}\Bigr)^{s - \frac{s}{\log(8s)}}  \cdot 2^s \leq \Bigl(\frac{1}{16}\Bigr)^{\frac{s}{2}} \cdot 2^s \leq
2^{-s}.\]
Second, for $\lfloor s/\log(8s) \rfloor + 1 < i \leq s - 1$, $(i - 1)/s \geq 1/2\log(8s)$ and so using $n \geq 2^{40s}$ and $s \geq 3$,
\[ (s - i) n^{\frac{i-1}{s}} \geq n^{\frac{1}{2\log(8s)}} \geq 2^{\frac{20s}{\log(8s)}} \geq 128s^3.\]
Therefore each term in the sum is at most
\[ \Bigl(\frac{(8s)^{\frac{s + i - 1}{s}}}{128s^3}\Bigr)^{s - i} {s \choose i} \leq \Bigl(\frac{(8s)^2}{128s^3}\Bigr)^{s-i} {s \choose i}
= \Bigl(\frac{1}{2s}\Bigr)^{s - i} {s \choose i}.\]
We conclude
\begin{eqnarray*}
\frac{\triangle(\chi)}{\mu(\chi)} &\leq& \sum_{i = 2}^{s - 1} \Bigl(\frac{1}{2s}\Bigr)^{s-i} {s \choose i}
+ \sum_{i = 2}^{s - 1} 2^{-s} \\
&\leq& \Bigl(1 + \frac{1}{2s}\Bigr)^{s} - 1 + (s - 2)2^{-s} \; \; \leq \; \; \sqrt{e} - 1 + (s - 2)2^{-s}.
\end{eqnarray*}
Evidently $(s - 2)2^{-s} \leq 1/8$ and therefore
\[ \sqrt{e} - 1 + (s - 2)2^{-s} \leq 0.773 \ldots < 1.\]
We conclude $\triangle(\chi) < \mu(\chi)$, as required. \qed

\end{document}